\documentclass[journal,twoside,web]{ieeecolor}
\usepackage{lcsys}

\usepackage[compress]{cite}
\usepackage{amsmath,amssymb,amsfonts}
\usepackage{graphicx}
\graphicspath{ {./graphics/} }

\usepackage{textcomp}
\usepackage{soul}
\usepackage{booktabs,makecell}
\usepackage[hidelinks]{hyperref}
\usepackage{enumerate}
\usepackage[capitalize,nameinlink]{cleveref}

\def\BibTeX{{\rm B\kern-.05em{\sc i\kern-.025em b}\kern-.08em
    T\kern-.1667em\lower.7ex\hbox{E}\kern-.125emX}}
\usepackage{mathtools}
\newcommand{\defeq}{\coloneq}

\newtheorem{thm}{Theorem}
\newtheorem{defn}{Definition}
\newtheorem{lem}{Lemma}

\newtheorem{prop}{Proposition}

\def\real{\mathbb{R}}
\def\complex{\mathbb{C}}
\def\torus{\mathbb{T}}

\def\integer{\mathbb{Z}}
\def\df{\nabla\!f}
\def\dft{\nabla\!\tilde f}
\newcommand{\tp}{\mathsf{T}}
\def\d{\textrm{d}}
\newcommand{\bmatl}[2]{\renewcommand{\arraystretch}{1.2}\left[\begin{array}{#1}#2\end{array}\right]}
\newcommand{\bmat}[1]{\begin{bmatrix}#1\end{bmatrix}}
\renewcommand{\epsilon}{\varepsilon}
\newcommand{\norm}[1]{\left\|#1\right\|}
\def\F{\mathcal{F}}

\def\QmL{\mathcal{Q}_{m,L}}
\def\SmL{\mathcal{S}_{m,L}^1}
\def\SSmL{\mathcal{S}_{m,L}^2}
\def\FmL{\mathcal{F}_{m,L}}

\title{The Fastest Known First-Order Method for Minimizing Twice Continuously Differentiable Smooth Strongly Convex Functions}

\author{Bryan Van Scoy, \IEEEmembership{Member, IEEE}, and
        Laurent Lessard, \IEEEmembership{Senior Member, IEEE}%
\thanks{
This material is based upon work supported by the National Science Foundation under Grant No. 2347121.
}
\thanks{B.~Van~Scoy is with the Department of Electrical and Computer Engineering, Miami University, Oxford, OH 45056, USA. (e-mail: bvanscoy@miamioh.edu)}
\thanks{L.~Lessard is with the Department of Mechanical and Industrial Engineering, Northeastern University, Boston, MA 02115, USA. (e-mail: l.lessard@northeastern.edu)}}

\begin{document}

\pagestyle{empty}
\maketitle
\thispagestyle{empty}

\begin{abstract}
We consider iterative gradient-based optimization algorithms applied to functions that are smooth and strongly convex. The fastest globally convergent algorithm for this class of functions is the Triple Momentum (TM) method. We show that if the objective function is also twice continuously differentiable, a new, faster algorithm emerges, which we call $C^2$-Momentum (C2M). We prove that C2M is globally convergent and that its worst-case convergence rate is strictly faster than that of TM, with no additional computational cost. We validate our theoretical findings with numerical examples, demonstrating that C2M outperforms TM when the objective function is twice continuously differentiable.
\end{abstract}


\section{INTRODUCTION}\label{sec:intro}

\IEEEPARstart{W}{e consider} the well-studied optimization problem
\begin{equation}\label{eq:opt}
    \underset{x\in \real^d}{\text{minimize}}\; f(x)
\end{equation}
where $f:\real^d\to\real$ is continuously differentiable.
A popular approach to solving \eqref{eq:opt}, particularly when the dimension $d$ is large, is to use iterative gradient-based methods, such as Gradient Descent (GD) and its accelerated variants.

A central question in the study of iterative methods is that of \emph{worst-case convergence rate} over a class of functions $\F$. In this letter, we consider the \emph{root-convergence factor} (also known as geometric convergence rate), denoted ${\rho\in(0,1)}$, a notion we make precise in \cref{sec:alg_description}. Associated with the root-convergence factor are two important concepts:

\paragraph{Lower bounds} $\rho$ is a \emph{lower bound} for $\F$ if for any algorithm, there exists $f\in\F$ and an algorithm initialization such that the algorithm converges no faster than~$\rho$.

\paragraph{Upper bounds} $\rho$ is an \emph{upper bound} for $\F$ if there exists an algorithm such that for all $f\in\F$ and algorithm initializations, the algorithm converges at least as fast as $\rho$.

If $\F$ has matching lower and upper bounds, this $\rho$ and the corresponding algorithm that achieves it are said to be \emph{minimax optimal} for $\F$.

Generally, adding more structure to a function class, such as convexity or Lipschitz properties, makes the minimax rate faster because iterative algorithms can exploit the additional structure to converge more rapidly. We now provide a brief survey of different function classes and their minimax rates. The relationship between these classes is illustrated in the Venn diagram of \cref{fig:venn}.

\begin{figure}[ht]
    \centering
    \includegraphics{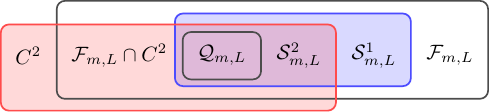}
    \caption{Venn diagram of different function classes. Blue region: strongly convex functions. Red region: twice continuously differentiable functions. This letter focuses on the shaded intersection of these sets, $\SSmL$.}
    \label{fig:venn}
\end{figure}

The class $\FmL$ consists of continuously differentiable functions with sector-bounded gradients. Specifically, there exists $x_\star\in\real^d$ (the optimal point) and constants $0<m\leq L$ such that
$\bigl(L(x-x_\star)-\df(x)\bigr)^\tp \bigl(\df(x)-m(x-x_\star)\bigr) \geq 0$ for all $x\in\real^d$. Functions in this class may be nonconvex but nevertheless have a unique local (and global) minimizer. The minimax rate for $\FmL$ is $\rho = \frac{\kappa-1}{\kappa+1}$ where $\kappa \defeq \frac{L}{m}$, and is achieved by GD with stepsize $\alpha=\frac{2}{L+m}$.

The class $\SmL$ consists of functions that have Lipschitz gradient with Lipschitz constant $L$ and are strongly convex with parameter $m$. The superscript ``1'' indicates that $f\in C^1$, which follows from Lipschitz gradients. One can show that $\SmL \subset \FmL$.
The minimax rate for $\SmL$ is $\rho = 1-\frac{1}{\sqrt{\kappa}}$, and was recently proved in \cite{drori_taylor_complexity} using an exact characterization of $\SmL$ via interpolation conditions and the Performance Estimation paradigm \cite{PEP}. The same lower bound was obtained in a parallel line of work by viewing algorithms as discrete-time Lur'e systems and applying integral quadratic constraints (IQCs) or dissipativity theory \cite{iqcopt,dissalg,scherer_optalg}. Specifically, the set $\SmL$ was over-approximated using Zames--Falb IQCs \cite{scherer_synthesis_siam,algosyn_acc}, leading to an upper bound that turned out to be exact.
The minimax rate for $\SmL$ is achieved by the Triple Momentum (TM) Method~\cite{tmm} and
the Information Theoretic Exact Method (ITEM)
\cite{ITEM}.

The class $\QmL \subset \SmL$ consists of quadratic functions of the form $f(x) = x^\tp Q x + p^\tp x + r$, with $m I_d \preceq Q \preceq L I_d$, and we have $\QmL\subset \SmL$.
The minimax rate for $\QmL$ is $\rho = \frac{\sqrt{\kappa}-1}{\sqrt{\kappa}+1}$. The lower bound was proved by Nemirovsky~\cite{nemirovsky1992information} and Nesterov \cite[\S2.1.4]{nesterov2018lectures}.
There are several minimax optimal methods for $\QmL$,
the simplest of which is  Polyak's Heavy Ball (HB) method \cite[\S3.2.1]{hb}.
Polyak used Lyapunov's indirect method to show that HB converges \emph{locally} for any
$f\in\SmL$ provided that $f$ is twice continuously differentiable ($f\in C^2$) in a neighborhood of the optimal point. In other words, HB converges on $\SmL$ 
when initialized sufficiently close to the optimal point and enjoys the same fast rate as for $\QmL$!
If incorrectly initialized, HB need not converge at all on $\SmL$ \cite{tmm,iqcopt}.

The aforementioned minimax optimal algorithms are described in \cref{sec:algform} and summarized in \cref{table:params}.

Polyak's observation raises an interesting possibility, which forms the starting point for the present work. If we consider the function class $\SSmL \defeq \SmL \cap C^2$, then by Lyapunov's indirect method, any globally convergent method will converge at its \emph{local rate}, which may be faster than the minimax rate of $\SmL$.
This function class satisfies
$\QmL \subset \SSmL \subset \SmL$ and may be characterized succinctly as functions whose Hessians satisfy $mI_d \preceq \nabla^2 f(x) \preceq LI_d$. Functions of interest in this category include
regularized logistic loss,
exponential family negative log-likelihoods with bounded natural parameters,
and Moreau envelope smoothing of any $f\in \SmL$.

Our main result is a new algorithm, $C^2$-Momentum (C2M). We show that C2M achieves
an upper bound of
$\max\left\{\frac{\sqrt{\kappa}-1}{\sqrt{\kappa}+1},\rho_\text{C2M}\right\}$ on $\SSmL$, where $\rho_\text{C2M}<1-\sqrt{\frac{2}{\kappa}}$.
This corresponds to an iteration complexity that is faster than the minimax rate of $\SmL$ by a factor of $\sqrt{2}$.

Notable related works are the recent papers \cite{Petersen1,Petersen2}, which use the same idea of optimizing the local convergence rate while enforcing global convergence. Specifically, these works develop re-tunings of HB and TM that converge globally on $\FmL$ but have optimized local rates because they also assume $f\in C^2$ locally near the optimal point.

The rest of this letter is organized as follows. In \cref{sec:alg_description} we describe C2M, in \cref{sec:convergence} we prove convergence results, in \cref{sec:numerical} we present some numerical results, and in \cref{sec:discussion} we discuss implications and future directions.

\section{MAIN RESULT}\label{sec:alg_description}

In this section, we describe our proposed algorithm, state its main convergence result, and use root locus arguments to provide intuition behind the algorithm parameters.

\subsection{Algorithm Form}\label{sec:algform}

We consider iterative first-order algorithms parameterized by $\alpha,\beta,\eta\in\real$ of the form
\begin{subequations}\label{eq:system}
    \begin{align}
        y_k &= x_k + \eta\,(x_k - x_{k-1}) \\
        u_k &= \df(y_k) \\
        x_{k+1} &= x_k + \beta\,(x_k-x_{k-1}) - \alpha\,u_k
    \end{align}
\end{subequations}
for $k\geq 0$ with initial conditions $x_0,x_{-1}\in\real^d$. We can interpret such an algorithm as a linear time-invariant (LTI) system $G$ in feedback with the gradient $\df$, where the transfer function\footnote{As a slight abuse of notation, we use the same symbol to refer to both an LTI system and its transfer function.} from the gradient $u_k$ to the point $y_k$ at which the gradient is evaluated is
\begin{equation}\label{eq:tf}
    G(z) = g(z) I_d \quad\text{where}\quad
    g(z) = -\alpha\frac{(1+\eta)z - \eta}{(z-1)(z-\beta)}.
\end{equation}
A minimal state-space realization of the reduced system $g$ is
\begin{equation}\label{eq:ss}
    \bmatl{c|c}{A & B \\ \hline C & 0} = \bmatl{cc|c}{1+\beta & -\beta & -\alpha \\ 1 & 0 & 0 \\ \hline 1+\eta & -\eta & 0}.
\end{equation}

Despite its simplicity, the form \eqref{eq:system} can represent \emph{all} algorithms referenced in \cref{sec:intro} when $\alpha,\beta,\eta$ are suitably chosen (GD, HB, TM, ITEM, GHB, GAG). \cref{table:params} shows parameters for the minimax methods discussed in \cref{sec:intro}.

\begin{table}[ht]
    \centering
    \renewcommand{\arraystretch}{1.5}
    \addtolength{\tabcolsep}{-1pt}
    \caption{Minimax-optimal methods for several function classes.}
    \label{table:params}
    \begin{tabular}{cccccc}
    \toprule
    \makecell{Function\\class} & \makecell{Minimax\\method} & $\alpha$ & $\beta$ & $\eta$ & \makecell{Minimax\\rate $\rho$} \\
    \midrule
    $\FmL$ & GD & $\frac{1-\rho}{m}$ & 0 & 0 & $\frac{\kappa-1}{\kappa+1}$ \\
    $\SmL$ & TM \cite{tmm} & $\frac{1+\rho}{L}$ & $\frac{\rho^2}{2-\rho}$ & $\frac{\rho^2}{(1+\rho)(2-\rho)}$ & $1-\frac{1}{\sqrt{\kappa}}$ \\
    $\QmL$ & HB \cite{hb} & $\frac{(1-\rho)^2}{m}$ & $\rho^2$ & 0 & $\frac{\sqrt{\kappa}-1}{\sqrt{\kappa}+1}$ \\
    \bottomrule
    \end{tabular}
\end{table}

\subsection{\texorpdfstring{$C^2$-Momentum}{C2M}}

\begin{defn}[C2M]
    Given parameters $m,L,\rho\in\real$ with $0<m\leq L$, $\kappa\defeq \frac{L}{m}$, and $\rho\in(0,1)$, the $C^2$-Momentum (C2M) algorithm is of the form \eqref{eq:system} with parameters
    \begin{equation}\label{eq:stepsizes}
        \begin{gathered}
        \alpha = \tfrac{(1-\rho)^2}{m}, \qquad
        \beta = \tfrac{\rho}{\kappa-1} \bigl(1 - \tfrac{\kappa\,(1 - 3\rho)}{1+\rho}\bigr), \\ 
        \eta = \tfrac{\rho}{\kappa-1}\bigl(\tfrac{1+\rho}{(1-\rho)^2} - \tfrac{\kappa}{1+\rho}\bigr).
        \end{gathered}
    \end{equation}
\end{defn}
\medskip

The C2M parameters depend on $\rho$, which we choose based on the condition number $\kappa$ of the objective function:
\begin{equation}\label{eq:rho_def}
    \begin{cases}
        \rho = \frac{\sqrt{\kappa}-1}{\sqrt{\kappa}+1} & \text{if }\kappa < 9+4\sqrt{5} \\[1.2mm]
        \rho \in \Bigl(\rho_\text{C2M},1-\sqrt{\frac{2}{\kappa}}\,\Bigr] & \text{if }\kappa\geq 9+4\sqrt{5}
    \end{cases}
\end{equation}
where $\rho_\text{C2M}$ is the smallest positive root of the polynomial
\begin{multline}\label{bigpoly}
    p(\kappa,\rho) \defeq 8 \kappa  (\kappa +1) \rho ^7
    -(23 \kappa ^2+18 \kappa +7) \rho ^6 \\
    +2 (5 \kappa ^2-14 \kappa -7) \rho ^5
    +(31 \kappa ^2+50 \kappa +15) \rho ^4 \\
    -4 (11 \kappa ^2-4 \kappa -11) \rho ^3
    +(23 \kappa ^2-30 \kappa +23) \rho ^2 \\
    -2 (\kappa -1) (3 \kappa +1) \rho
    +(\kappa -1)^2.
\end{multline}
When $\kappa < 9+4\sqrt{5}$, the parameters of C2M reduce to those of HB \cite{hb} in \Cref{table:params}.
For $\kappa\geq 9+4\sqrt{5}$, we in general want to pick $\rho$ as small as possible, but we will see that proving global asymptotic stability requires a strict inequality, so in practice we can choose $\rho = \rho_\text{C2M} + \epsilon$ for some small $\epsilon > 0$.
The C2M stepsizes are defined in terms of the root $\rho_\text{C2M}$ of the polynomial $p(\kappa,\rho)$ in~\eqref{bigpoly}. The following result (i) shows that this quantity is well defined in that the polynomial does have a positive root, and (ii) provides bounds on this root that will be used in the analysis. The proof is in \cref{app:sturm}.

\begin{lem}\label{lem:sturm}
    Suppose $\kappa \geq 9+4\sqrt{5}$. The polynomial $p(\kappa,\rho)$ defined in \eqref{bigpoly} has exactly one real root $\rho_\text{C2M}$ in the open interval $\Bigl(\frac{\sqrt{\kappa}-1}{\sqrt{\kappa}+1}, 1-\sqrt{\frac{2}{\kappa}}\Bigr)$. Moreover, $\rho_\text{C2M}$ is the smallest positive root and $p(\kappa,\rho) < 0$ for all $\rho\in\Bigl(\rho_\text{C2M},1-\sqrt{\frac{2}{\kappa}}\,\Bigr]$.
\end{lem}

\subsection{Main Result}

To describe our main result, we first define the root-convergence factor of an algorithm, which is a way to characterize its rate of convergence; see \cite[\S9.2]{ortega1970}.

\begin{defn}\label{def:r_conv}
    Let $\{x_k\}$ be a sequence that converges to a point $x_\star$. Then, the root-convergence factor of $\{x_k\}$ is
    \[
        \rho = \limsup_{k\to\infty} \, \norm{x_k-x_\star}^{1/k}.
    \]
    Moreover, the worst-case root-convergence factor of an algorithm \eqref{eq:system} over a function class $\F$ is the supremum of the root-convergence factors over all sequences produced by the algorithm when applied to a function $f\in\F$.
\end{defn}

We now state our main convergence result for C2M over the class $\SSmL$. A full proof is included in \cref{sec:convergence}.

\begin{thm}[Upper bound for C2M]\label{thm:c2m}
    Consider the C2M method defined in \eqref{eq:stepsizes} with parameter $\rho$ chosen according to \eqref{eq:rho_def}. 
    An upper bound for the worst-case root-convergence factor of C2M over the function class $\SSmL$ is $\rho$.
\end{thm}

\subsection{Root Locus Interpretation}\label{sec:rootlocus}

\begin{figure}[t]
    \centering
    \includegraphics{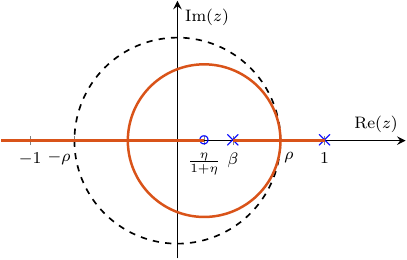}
    \caption{Root locus of C2M. The locus has a double root at $z=\rho$ at gain $m$ and a single root at $z=-\rho$ at gain $L$.}
    \label{fig:rootlocus}
\end{figure}

Before rigorously analyzing the convergence of C2M, we first provide intuition behind the C2M parameters \eqref{eq:stepsizes} using a root locus argument.

Consider the general algorithm \eqref{eq:system} applied to a function $f\in\QmL \subset \SSmL$ with Hessian $Q$. By diagonalizing the Hessian, the iterates separate into $d$ decoupled systems, each in (positive) feedback with an eigenvalue $q_i$ of $Q$. Since the objective function is $L$-smooth, $m$-strongly convex, and twice continuously differentiable, its Hessian has eigenvalues in the interval $[m,L]$. Therefore, we can study worst-case local convergence by analyzing the eigenvalues of $A + q B C$ for $q\in[m,L]$. These closed-loop eigenvalues are solutions of the root locus $0 = 1 - q\,g(z)$ for $q\in[m,L]$. The parameters of C2M are the solutions to the following conditions:
\begin{enumerate}
    \item The root locus passes through $z=-\rho$ when $q=L$.
    \item The root locus has a double root at $z=\rho$ when $q=m$.
\end{enumerate}
The visual reasoning for these two conditions is illustrated in \cref{fig:rootlocus}, which shows the root locus of $1-q\,g(z)$ as $q$ varies. As $q\to 0$, the roots are the poles of $g(z)$, which are $\beta$ and~$1$. These roots meet at $z=\rho$, circle around the zero at $z = \tfrac{\eta}{1+\eta}$, then break in on the negative real axis, with one root converging to the zero and the other going to $-\infty$ along the real axis. By enforcing the above two conditions, the root locus remains entirely inside the $\rho$-circle for all $q\in[m,L]$. In terms of the transfer function, these conditions are that
\begin{align}\label{eq:rootlocus-conditions}
    L\,g(-\rho) &= 1, &
    m\,g(\rho) &= 1, &
    \left.\frac{\d g(z)}{\d z} \right|_{z=\rho} \!\!\!= 0,
\end{align}
where the last two equations are for the double root. Straightforward calculations show that the parameters \eqref{eq:stepsizes} for C2M are the unique solution to the equations \eqref{eq:rootlocus-conditions}.

\section{CONVERGENCE ANALYSIS}\label{sec:convergence}

We now prove the main convergence result for C2M from \cref{thm:c2m}. Our proof consists of two steps. First, we show that the algorithm is globally asymptotically stable, meaning that the iterates converge to the minimizer of $f$ for all initial conditions. Once we have global convergence, we then show that the worst-case root-convergence factor is $\rho$ by analyzing the linearization of the algorithm about its equilibrium.

\subsection{Global Stability via Frequency-Domain Analysis}

It is convenient to shift the dynamics of the algorithm \eqref{eq:system} about its optimal point $x_\star$, which satisfies $\df(x_\star)=0$. To this effect, define $\tilde x_k \defeq x_k-x_\star$, $\tilde y_k \defeq y_k-x_\star$, $\tilde u_k = u_k$, and $\tilde f(y) \defeq f(y+x_\star)$. Then, we can rewrite \eqref{eq:system} as:
\begin{subequations}\label{eq:system_shifted}
    \begin{align}
        \tilde y_k &= \tilde x_k + \eta\,(\tilde x_k - \tilde x_{k-1}) \\
        \tilde u_k &= \dft(\tilde y_k) \\
        \tilde x_{k+1} &= \tilde x_k + \beta\,(\tilde x_k-\tilde x_{k-1}) - \alpha\,\tilde u_k
    \end{align}
\end{subequations}
Convergence of the algorithm $G$ applied to $f$ is therefore equivalent to convergence of $G$ applied to $\tilde f$. In other words, we may assume without loss of generality that $x_\star=0$.

To verify global asymptotic stability, we use integral quadratic constraints (IQCs)~\cite{megretski_rantzer}. In discrete time, these are defined as follows (see~\cite{fetzer2017absolute}), where $\ell_2^{n}$ denotes the space of square-summable sequences on $\real^n$.

\begin{defn}
    Signals $y\in\ell_2^{n_y}$ and $u\in\ell_2^{n_u}$ with associated $z$-transforms $\hat y(z)$ and $\hat u(z)$ satisfy the IQC defined by a
    measurable, bounded, and Hermitian matrix-valued 
    function $\Pi : \torus\to\complex^{(n_y+n_u)\times(n_y+n_u)}$ if
    \begin{equation}\label{iqc}
        \int_\torus \bmat{\hat y(z) \\ \hat u(z)}^* \Pi(z) \bmat{\hat y(z) \\ \hat u(z)} \d z \geq 0,
    \end{equation}
    where $\torus \defeq \left\{ z \in \complex \;\middle|\; |z|=1\right\}$ is the unit circle in the complex plane. A bounded operator $\Delta : \ell_2^{n_y}\to\ell_2^{n_u}$ satisfies the IQC defined by $\Pi$ if \eqref{iqc} holds for all $y\in\ell_2^{n_y}$ with $u = \Delta(y)$.
\end{defn}

It is well known (see for example \cite{zames-falb-qp,fetzer2017absolute}) that the gradient of a smooth strongly convex function can be described using IQCs.

\begin{prop}
    The operator $\Delta : \ell_2^d \to \ell_2^d$ defined by $(\Delta(y))_k \defeq \dft(y_k)$ for all $k\geq0$ and $y\in\ell_2^d$, where $\tilde f\in\SmL$ and $\dft(0) = 0$ satisfies the O'Shea--Zames--Falb IQC $\Pi_{m,L}\otimes I_d$, where\looseness=-1
    \[
        \Pi_{m,L} \defeq \bmat{-mL(2-h-h^*) & \!\! L(1-h^*) + m(1-h) \\ L(1-h) + m(1-h^*)\!\! & -(2-h-h^*)}
    \]
    and $h(z)$ is any transfer function with impulse response $\{h_k\}$ satisfying $\norm{h}_1 = \sum_{k=-\infty}^\infty |h_k| \leq 1$ and $h_k\geq 0$ for all $k$.
\end{prop}

While $\Delta$ satisfies the IQC $\Pi_{m,L}\otimes I_d$, to analyze the interconnection of the algorithm $G$ with $\Delta$ using the main IQC theorem (see \cref{thm:iqc}), we will first need to perform a loop transformation (see, e.g., \cite[\S 6.6]{khalil}) so that
the zero operator is contained in the class of transformed uncertainties.
Doing so, the feedback interconnection of $G$ and $\Delta$ is equivalent to the feedback interconnection of $\tilde G$ and $\tilde \Delta$, where
\[
    \tilde G(z) = \tilde g(z)\otimes I_d \quad\text{where}\quad
    \tilde g(z) = \frac{\frac{L-m}{2} g(z)}{1 - \frac{L+m}{2} g(z)}
\]
and the transformed operator is given by
\[
    \tilde \Delta(x) = \tfrac{2}{L-m} \bigl(\Delta(x) - \tfrac{L+m}{2}x\bigr).
\]

Using properties of shifting and scaling the gradient of smooth strongly convex functions~\cite[\S 2.4]{PEP},
$\tilde \Delta$ satisfies the IQC $\Pi_{-1,1}\otimes I_d$ if and only if $\Delta$ satisfies the IQC $\Pi_{m,L}\otimes I_d$.
We are now ready to apply the following main IQC result.

\def\citeFetzer{\cite[Thm.~2]{fetzer2017absolute}}
\begin{prop}[Discrete-time IQC result \citeFetzer]\label{thm:iqc}
    Fix $\rho\in(0,1)$. Suppose that $\tilde G$ is stable, $\tilde \Delta$ is a bounded causal operator, and
    \begin{enumerate}[(i)]
        \item the interconnection of $\tilde G$ and $\tilde \Delta$ is well-posed,
        \item for every $\tau\in[0,1]$, $\tau\tilde \Delta$ satisfies the IQC $\Pi$, and
        \item the following frequency-domain inequality holds:
        \[
            \bmat{\tilde G(z) \\ I}^* \Pi(z) \bmat{\tilde G(z) \\ I} < 0 \quad \text{for all }z\in\torus.
        \]
    \end{enumerate}
    Then the feedback interconnection of $\tilde G$ and $\tilde \Delta$ is stable.
\end{prop}

Applying \cref{thm:iqc} therefore yields the following.

\begin{prop}\label{thm:fdi}
     Algorithm~\eqref{eq:system} is globally asymptotically stable for all $f\in\SmL$ if
     $\bigl(1-\tfrac{L+m}{2} g(z)\bigr)^{-1}$ is stable and
     the following frequency-domain inequality (FDI) holds:
    \begin{equation}\label{FDI}
        \bmat{g(z) \\ 1}^* \Pi_{m,L}(z) \bmat{g(z) \\ 1} < 0
        \quad\text{for all } z\in\torus,
    \end{equation}
    where $h(z)$ satisfies $\norm{h}_1\leq 1$ and $h_k\geq 0$ for all~$k\in\integer$.
\end{prop}

\begin{proof}
    The stability condition is equivalent to stability of~$\tilde G$. It is straightforward to verify that the interconnection of $\tilde G$ and $\tilde\Delta$ is well-posed and that $\tau\tilde\Delta$ satisfies the IQC $\Pi = \Pi_{-1,1}\otimes I_d$ for all $\tau\in[0,1]$. Therefore, the first two conditions in \cref{thm:iqc} hold for the transformed system $\tilde G$ and the IQC $\Pi$. It remains to show that the FDI in (iii) is equivalent to that in \eqref{FDI}. To that end, we first write the numerator and denominator of $\tilde g$ as
    \[
        \bmat{\tfrac{L-m}{2} g \\ 1 - \tfrac{L+m}{2} g} = \bmat{\frac{L-m}{2} & 0 \\ -\frac{L+m}{2} & 1} \bmat{g \\ 1} = M \bmat{g \\ 1}.
    \]
    Using this relationship along with $M^\tp \Pi_{-1,1} M = \Pi_{m,L}$, the FDI in (iii) of \cref{thm:iqc} is
    \[
        \bmat{\tilde g \\ 1}^* \Pi_{-1,1} \bmat{\tilde g \\ 1}
        = \frac{1}{|1-\tfrac{L+m}{2} g|^2} \bmat{g \\ 1}^* M^\tp \Pi_{-1,1} M \bmat{g \\ 1}.
    \]
    Therefore, the FDI in (iii) is equivalent to that in~\eqref{FDI}. 
    From \cref{thm:iqc}, the interconnection of $\tilde G$ and $\tilde \Delta$ is stable, which via loop shifting implies the interconnection of $G$ and $\Delta$ is stable. Finally, (input-output) stability means that all signals have bounded norms. Therefore, $\norm{\tilde x} <\infty \implies \lim_{k\to\infty} \tilde x_k = 0 \implies \lim_{k\to\infty} x_k = x_\star$.
\end{proof}

We use \cref{thm:fdi} with $h(z) = z^{-1}$ to show that C2M is convergent by directly verifying the FDI \eqref{FDI}. First, the condition that $\bigl(1-\tfrac{L+m}{2} g(z)\bigr)^{-1}$ is bounded follows from the root locus argument in \cref{sec:rootlocus}; see the following \cref{sec:local} for a more rigorous argument. Letting $z = x + i \sqrt{1-x^2}$ for $x\in[-1,1]$ and substituting the C2M parameters, it is straightforward to verify that the FDI is satisfied when $\rho=\frac{\sqrt{\kappa}-1}{\sqrt{\kappa}+1}$ and $\kappa < 9+4\sqrt{5}$. In the other case with $\kappa > 9+4\sqrt{5}$, the FDI reduces to the inequality
\begin{multline}\label{FDI2}
    0 > -4 \rho \left(\kappa(1-\rho)^2-(1+\rho)\right)x^2  \\
    -2 (1-\rho) \left(\kappa (1-\rho)^2(1+2\rho)-(1+\rho)^2\right)x\\
    -(1+\rho) \left(\kappa (1-\rho)^2(1-4\rho+\rho^2)+6 \rho-2 \rho^3\right).
\end{multline}
The right-hand side of \eqref{FDI2} is a quadratic in $x$. To show that this inequality holds, we will use the following.

\begin{lem}\label{lem:ineqs}
    Suppose $\rho \in [0,1]$ and $\kappa > 1$. Then,
    \begin{align*}
        \rho > \frac{\sqrt{\kappa}-1}{\sqrt{\kappa}+1} \quad&\iff\quad \kappa(1-\rho)^2 < (1+\rho)^2, \\
        \rho < 1-\sqrt{\frac{2}{\kappa}}\quad&\iff\quad \kappa(1-\rho)^2 > 2.
    \end{align*}
\end{lem}

Since $\rho < 1-\sqrt{\tfrac{2}{\kappa}}$ by assumption, it follows from \cref{lem:ineqs} that $\kappa(1-\rho)^2 > 2 > 1+\rho$. Therefore, the leading coefficient of the quadratic is negative. Maximizing the right-hand side of \eqref{FDI2}, this inequality holds if
\[
    \frac{p(\kappa,\rho)}{4\rho(\kappa(1-\rho)^2-(1+\rho))} < 0,
\]
where $p(\kappa,\rho)$ is the polynomial in \eqref{bigpoly}. The denominator is positive from the prior argument. Moreover, $p(\kappa,\rho)$ is negative for any $\rho \in \bigl(\rho_\text{C2M},1-\sqrt{\tfrac{2}{\kappa}}\bigr]$ by \cref{lem:sturm}, so the FDI is satisfied. Therefore, C2M is globally asymptotically stable for any $\rho$ satisfying~\eqref{eq:rho_def}.

\subsection{Local Convergence}\label{sec:local}

While the root locus interpretation
provides intuition behind the local convergence of C2M, we now use Lyapunov's indirect method along with the Jury criterion to systematically prove local convergence; see~\cite{mohammadi2025tradeoffs} for similar analyses in other settings. We begin by characterizing the worst-case root convergence factor in terms of the system matrices.

\begin{lem}\label{lem:spectral-radius}
    The worst-case root-convergence factor of the algorithm~\eqref{eq:system} over the function class $\SSmL$ is the maximum spectral radius of $A+qBC$ over $q\in[m,L]$.
\end{lem}

\begin{proof}
    From the linear convergence theorem~\cite[Thm. 10.1.4]{ortega1970}, the root-convergence factor of the algorithm~\eqref{eq:system} is the spectral radius of its linearization evaluated at the equilibrium. In particular, let $\{x_k\}$ denote the sequence produced by applying algorithm~\eqref{eq:system} to a function $f\in\SSmL$ for some initial conditions $x_0,x_{-1}\in\real^d$. Let $Q$ denote the Hessian of $f$ evaluated at the minimizer of $f$. Then the root-convergence factor of the sequence $\{x_k\}$ is the spectral radius of the linearization $A\otimes I_d + (BC)\otimes Q$, where $\otimes$ denotes the Kronecker product. Since $Q$ is real and symmetric, it is diagonalizable. Applying this diagonalization to the linearized system yields $A\otimes I_d+(BC)\otimes\text{diag}(q_1,\ldots,q_d)$, where $q_1,\ldots,q_d$ are the eigenvalues of $Q$. Therefore, the worst-case root-convergence factor over the function class $\SSmL$ is the maximum spectral radius of $A+qBC$ over $q\in[m,L]$.
\end{proof}

Based on \cref{lem:spectral-radius}, we can characterize the worst-case root-convergence factor using the eigenvalues of $A+qBC$. We next analyze these eigenvalues using the Jury criterion.
Recall that a polynomial $z^2 + a_1\,z + a_0$ with real coefficients has roots in the closed unit disk if and only if \cite[\S 4.5]{fadali-2009}\footnote{The reference states the results for the roots to be contained in the \textit{open} unit disk, which is described by the corresponding \textit{strict} inequalities. Since the roots of a polynomial depend continuously on its coefficients, the corresponding result for the closed unit disk holds with non-strict inequalities.}
\begin{equation}\label{eq:jury}
    1 + a_1 + a_0 \geq 0, \quad 1 - a_1 + a_0 \geq 0, \quad |a_0| \leq 1.
\end{equation}
The characteristic polynomial of the closed-loop system matrix $A+qBC$ is the quadratic $\chi(z) = z^2 + (q\alpha(1+\eta)-(1+\beta)) z + (\beta-q\alpha\eta)$. Applying the Jury criterion~\eqref{eq:jury} to the scaled polynomial $\chi(\rho z)$, the closed-loop eigenvalues are in the closed $\rho$-disk if and only if
\begin{gather*}
    (1-\rho)(\beta-\rho) + \alpha\,(\eta\rho - \eta + \rho) q \geq 0, \\
    (1+\rho)(\beta+\rho) - \alpha\,(\eta\rho + \eta + \rho) q \geq 0, \\
    \rho^2 + \beta - \alpha\eta q \geq 0, \qquad \rho^2 - \beta + \alpha\eta q \geq 0,
\end{gather*}
for all $q\in[m,L]$. Since each inequality is linear in $q$, it suffices to enforce the inequality at the endpoints ${q\in\{m,L\}}$. Substituting the C2M parameters, this system of inequalities reduces to
\[
    \frac{\sqrt{\kappa}-1}{\sqrt{\kappa}+1} \leq \rho \leq \frac{\kappa-1}{\kappa+1}.
\]
The lower bound on $\rho$ is the minimax rate for~$\QmL$.
Since all parameters $\rho$ in~\eqref{eq:rho_def} satisfy these conditions, we have that all eigenvalues of $A+qBC$ are in the $\rho$-disk. Therefore, from \cref{lem:spectral-radius}, the parameter $\rho$ is the worst-case root-convergence factor of C2M, which completes the proof of \cref{thm:c2m}.

\subsection{Iteration Complexity}

It is common in optimization to characterize algorithm convergence using \emph{iteration complexity} \cite[\S1.1.2]{nesterov2018lectures}. Iteration complexity is an expression for how the worst-case number of iterations $N$ required to reach a specified error $\epsilon$ scales as a function of problem parameters such as $\kappa$, expressed asymptotically as $\epsilon\to 0$ and $\kappa\to\infty$.
If the convergence rate is $\rho$ as defined in \cref{def:r_conv}, then $\norm{x_k-x_\star} \leq c(k) \rho^k$, where $c(k)$ grows sub-exponentially in $k$. We seek the smallest $N$ such that $c(N) \rho^N \leq \epsilon$. Rearranging, we obtain $\log c(N) + N \log\rho \leq \log\epsilon$. Since $c(N)$ is sub-exponential, it is dominated by the linear term in $N$ as $\epsilon\to 0$ (and therefore $N\to\infty$), so we neglect it. We are left with $N \geq \frac{-1}{\log \rho}\log\frac{1}{\epsilon}$. Next, we expand $\frac{-1}{\log\rho}$ as a function of $\kappa\to\infty$, keeping only the most significant term. For example, if $\rho = 1-\frac{c}{\sqrt{\kappa}}$,
\[
   \frac{-1}{\log\rho} = \frac{-1}{\log\bigl(1-\frac{c}{\sqrt\kappa}\bigr)}
    = \frac{\sqrt\kappa}{c} -\frac{1}{2}-\frac{c}{12\sqrt\kappa}
    -\cdots \approx \frac{\sqrt\kappa}{c}.
\]
Based on the minimax rate of TM in \cref{table:params} ($c=1$), we conclude that
$N_\text{TM} \gtrsim \sqrt\kappa \log\tfrac{1}{\epsilon}$. Similarly, for HB, we have
$\frac{\sqrt\kappa-1}{\sqrt\kappa+1} = 1-\frac{2}{\sqrt\kappa+1} \approx 1-\frac{2}{\sqrt\kappa}$ ($c=2$), so $N_\text{HB} \gtrsim \frac{\sqrt\kappa}{2} \log\tfrac{1}{\epsilon}$.

For C2M, we do not have a nice expression for $\rho_\text{C2M}$, but we can nevertheless find an asymptotic analytic expansion for it about $\kappa\to\infty$, which leads to the bounds
\begin{equation*}
    1-\sqrt{\frac{2}{\kappa}}-\frac{1+2\sqrt{2}}{4\kappa}
    <
    \rho_\text{C2M}
    <   
    1-\sqrt{\frac{2}{\kappa}}.
\end{equation*}
Therefore, $c=\sqrt{2}$ and 
$N_\text{C2M} \gtrsim \frac{\sqrt\kappa}{\sqrt{2}} \log\tfrac{1}{\epsilon}$.
In other words, C2M is faster than TM by a factor of $\sqrt{2}$.

In contrast, GD has iteration complexity $N_\text{GD} \gtrsim \tfrac{\kappa}{2} \log\frac{1}{\epsilon}$. In the optimization literature, methods with the $\sqrt{\kappa}$ factor instead of merely $\kappa$ are called \emph{accelerated methods}. We can visualize iteration complexity by plotting $\frac{-1}{\log\rho}$ versus $\kappa$ on a log-log scale (we omit the $\log\frac{1}{\epsilon}$ factor); see \cref{fig:iteration-complexity}.
We also included a plot for GD (see \cref{table:params}).

We see in \cref{fig:iteration-complexity} that non-accelerated methods (GD, GAG) have an asymptotic slope of $1$ whereas accelerated methods (C2M, TM) have an asymptotic slope of $\frac12$.

\begin{figure}[ht]
    \centering
    \includegraphics[width=\linewidth]{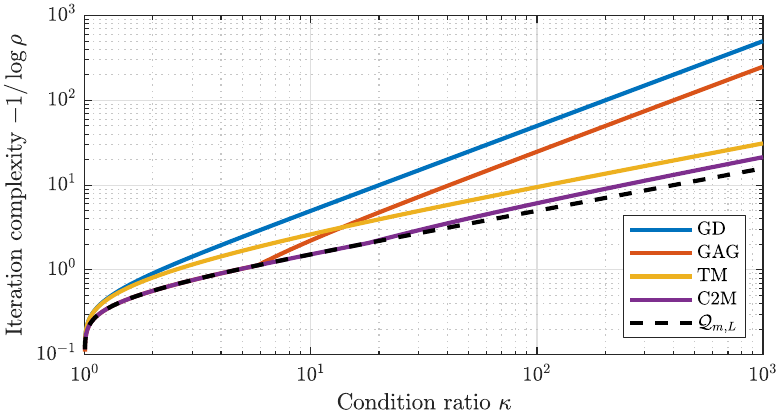}
    \caption{Iteration complexity of several iterative methods applied to $\SSmL$. The proposed C2M method outperforms TM \cite{tmm}, which is minimax optimal on $\SmL$, by exploiting a faster local convergence rate. Similarly, GAG \cite{Petersen2} outperforms GD, which is minimax optimal on $\FmL$.}
    \label{fig:iteration-complexity}
\end{figure}

\section{NUMERICAL VALIDATION}\label{sec:numerical}

We simulate our proposed algorithm C2M along with several other first-order methods on a function chosen to showcase worst-case behavior. We used the function \cite[\S IV]{tmm}
\[
    f(x) = (L-m)\sum_{i=1}^p g(a_i^\tp x-b_i)+\frac{m}{2}\norm{x}^2,
\]
where $g(w)$ is $\frac{1}{2}w^2 e^{-r/w}$ if $w>0$ and zero if $w\leq 0$.
When $r>0$ and $0<m\leq L$ and
${\norm{\bmat{a_1 & \cdots & a_p}}=1}$, such functions satisfy $f\in \SSmL$.
We chose the parameters $L=1$, $m=10^{-3}$, $r=10^{-3}$, $p=2$,
$a_1=(1,0)$, $a_2=(0,0.002)$, and $b_1=b_2=100$. All methods were initialized at $x_0=0$.

In \cref{fig:simulations}, we plot error as a function of iteration. The function $f$ elicits worst-case behavior from GD, HB, and TM. In other words, GD and TM converge at their respective minimax rates for $\FmL$ and $\SmL$. Since $f \notin \QmL$, HB is only locally convergent. In our simulation, we see that HB does not converge; however, if we were to initialize HB sufficiently close to $x_\star$, then it would converge at least as fast as the minimax rate for $\QmL$.
Our proposed C2M exploits additional smoothness in the objective to converge globally at a rate that is always faster than the minimax $\SmL$ rate.
Likewise, GAG, which is globally convergent on $\FmL$, is slightly faster than GD, which is minimax-optimal on $\FmL$.

\begin{figure}[ht]
    \centering
    \includegraphics[width=\linewidth]{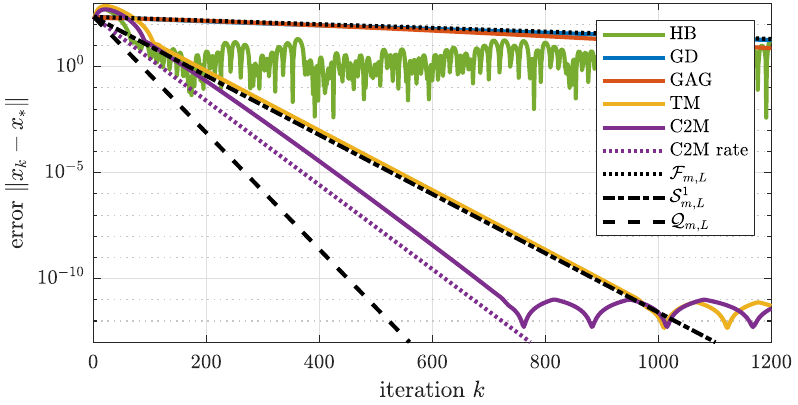}
    \caption{Simulation results for a function $f\in\SSmL$ (see \cref{sec:numerical}). Solid lines are simulation results for the specified method; black lines are minimax rates for different function classes (see \cref{table:params}); the dotted purple line is our theoretical upper bound (worst-case) rate for C2M.
    }
    \label{fig:simulations}
\end{figure}

\vspace{-2mm}

\section{DISCUSSION}\label{sec:discussion}

The proposed C2M algorithm is the first method, to the best of the authors' knowledge, that is designed specifically for the function class $\SSmL$. The minimax rate for this function class, however, is not known, in contrast to the function classes $\SmL$ and $\QmL$. Finding this minimax rate or even lower bounds are interesting open problems.

The parameters of C2M are related to two other algorithms from the literature. As we have already seen, C2M reduces to HB when $\rho=\frac{\sqrt{\kappa}-1}{\sqrt{\kappa}+1}$. Moreover, the general C2M parameters are identical (after appropriate transformations) to those of GAG~\cite[Cor. 1.1]{Petersen2}. This makes sense, since the work \cite{Petersen2} also considers the family of algorithms \eqref{eq:system} and is optimizing for local convergence.  
The two cases differ, however, in the choice of $\rho$, since GAG is optimized over the function class $\FmL$ defined in \cref{sec:intro} rather than $\SSmL$.

\appendix 

\subsection{\texorpdfstring{Proof of \cref{lem:sturm}}{Proof of Lemma}}\label{app:sturm}

We apply Sturm's theorem \cite[Thm. 2.62]{basu2007algorithms} to $p(\kappa,\rho)$ as a polynomial in $\rho$. Define the Sturm sequence
\[
    p_0 = p, \quad p_1 = \frac{\d p}{\d\rho}, \quad p_{i+1} = -\text{rem}(p_{i-1},p_i) \text{ for }i\geq 1,
\]
where $\text{rem}()$ denotes the remainder after polynomial division (considered as polynomials in $\rho$), and the sequence terminates when $p_i$ is constant, which occurs for $i\leq 7$ since $p$ is degree~$7$ in~$\rho$. Evaluating the Sturm sequence at $\rho=0$ and $\rho=1$ yields 5 sign changes and 3 sign changes, respectively. Therefore, there are two real roots in the interval $(0,1)$. Moreover, $p$ is positive when $\rho = \frac{\sqrt{\kappa}-1}{\sqrt{\kappa}+1}$, negative when $\rho = 1-\sqrt{\tfrac{2}{\kappa}}$, and positive when $\rho = 1$. By the intermediate value theorem, we conclude that there is exactly one real root in each interval $\bigl(\frac{\sqrt{\kappa}-1}{\sqrt{\kappa}+1},1-\sqrt{\tfrac{2}{\kappa}}\bigr)$ and $\bigl(1-\sqrt{\tfrac{2}{\kappa}},1\bigr)$, and the value of $p$ is negative for all $\rho\in\bigl(\rho_\text{C2M},1-\sqrt{\tfrac{2}{\kappa}}\bigr]$.\hfill\QED

\bibliographystyle{IEEEtran}
\bibliography{refs_abbrv}

\end{document}